\documentclass{article}
\usepackage{fullpage,nicefrac,comment}
\usepackage{amssymb,amsmath,amsthm,amsfonts,enumitem} 
\usepackage[margin=1.125in]{geometry}
\usepackage{xcolor}
\definecolor{DarkGreen}{rgb}{0.1,0.5,0.1}
\definecolor{DarkRed}{rgb}{0.5,0.1,0.1}
\usepackage[T1]{fontenc}
\usepackage{amsbsy}
\definecolor{DarkBlue}{rgb}{0.1,0.1,0.5}
\usepackage{algorithm}
\usepackage{algpseudocode}
\usepackage[small]{caption}
\usepackage[pdftex]{hyperref}
\hypersetup{
    unicode=false,          
    pdftoolbar=true,        
    pdfmenubar=true,        
    pdffitwindow=false,      
    pdfnewwindow=true,      
    colorlinks=true,       
    linkcolor=DarkBlue,          
    citecolor=DarkGreen,        
    filecolor=DarkGreen,      
    urlcolor=DarkBlue,          
    pdftitle={},
    pdfauthor={},
}

\usepackage{bm,bbm,amsmath,amssymb,amsthm,graphicx}

\numberwithin{equation}{section}

\newtheorem{theorem}{Theorem}[section]
\newtheorem{corollary}[theorem]{Corollary}
\newtheorem{lemma}[theorem]{Lemma}
\newtheorem{definition}[theorem]{Definition}

\def \trace {\operatorname*{trace}}
\def \zero {\mathbf{0}}
\def \R {\mathbb{R}}
\def \W {\mathbf{W}}
\def \X {\mathbf{X}}
\def \A {\mathcal{A}}
\def \H {\mathcal{H}}
\def \Id {\mathcal{I}}
\def \C {\mathbb{C}}
\def \Z {\mathbf{Z}}
\def \Y {\mathbf{Y}}
\def \y {\bm{y}}
\def \e {\bm{e}}
\def \z {\bm{z}}
\def \x {\bm{x}}

\DeclareMathOperator*{\argmin}{argmin}

\title{Iterative Hard Thresholding for Low CP-rank Tensor Models}

\author{Rachel Grotheer, Shuang Li, Anna Ma, Deanna Needell, Jing Qin}
\date{\today}

\def \X				{\boldsymbol{X}}
\def \Y				{\boldsymbol{Y}}

\begin{document}
\maketitle
\begin{abstract}
Recovery of low-rank matrices from a small number of linear measurements is now well-known to be possible under various model assumptions on the measurements. Such results demonstrate robustness and are backed with provable theoretical guarantees. However, extensions to tensor recovery have only recently began to be studied and developed, despite an abundance of practical tensor applications. Recently, a tensor variant of the Iterative Hard Thresholding method was proposed and theoretical results were obtained that guarantee exact recovery of tensors with low Tucker rank. In this paper, we utilize the same tensor version of the Restricted Isometry Property (RIP) to extend these results for tensors with low CANDECOMP/PARAFAC (CP) rank. In doing so, we leverage recent results on efficient approximations of CP decompositions that remove the need for challenging assumptions in prior works. We complement our theoretical findings with empirical results that showcase the potential of the approach.
\end{abstract}

\section{Introduction}

The field of compressive sensing \cite{RefWorks:45,RefWorks:373} has lead to a rich corpus of results showcasing that intrinsically low-dimensional objects living in large ambient dimensional space can be recovered from small numbers of linear measurements. As a complement to the so-called sparse vector recovery problem is the low-rank matrix recovery problem. Motivated by applications in signal processing (see e.g. \cite{ahmed2014compressive,zhang2013hyperspectral,gross2010quantum,candes2011robust}) and data science (see e.g. \cite{basri2003lambertian,candes2009exact}), the latter asks for a low-rank matrix to be recovered from a small number of linear measurements or observations. Formally, one considers a matrix $\X\in\R^{n_1\times n_2}$ of (nearly) rank $r \ll \min(n_1, n_2)$ along with a linear operator $\A : \R^{n_1\times n_2} \rightarrow \R^m$ with $m \ll n_1n_2$, and the goal is to recover $\X$ from the measurements $\y = \A(\X)$. A common approach is to consider the relaxation of the NP-hard rank minimization \cite{eldar2012uniqueness}, leading to the so-called nuclear-norm minimization problem \cite{candes2009exact,recht2010guaranteed},
\begin{equation}\label{eq:NNM}
\hat{\X} = \argmin_{\X\in\R^{n_1\times n_2}} \|\X\|_* \quad\text{subject to}\quad \A(\X) = \y,
\end{equation}
where the nuclear-norm is the L1 norm of the singular values: $\|\X\|_* := \sum_i \sigma_i(\X) = \trace(\sqrt{\X^*\X})$. It has been shown that \eqref{eq:NNM} yields exact recovery of $\X$ (or approximate when the measurements contain noise) when the measurement operator $\A$ obeys some assumptions such as incoherence, restricted isometry, or is constructed from some random models \cite{candes2010tight,candes2009exact,recht2010guaranteed}. Examples include when $\A$ is constructed by taking matrix inner products with matrices containing i.i.d. (sub)-Gaussian entries or when $\A$ views entries of the matrix $\X$ selected uniformly at random. In these and most other cases, the number of measurements required is on the order of $m\approx r\max(n_1, n_2)$.

As in vector sparse recovery, an alternative to optimization based programs like \eqref{eq:NNM} is to use iterative methods that produce estimates that converge to the solution $\X$. Relevant to this paper is the Iterative Hard Thresholding (IHT) method \cite{PaperIHT,blumensath2010normalized,tanner2013normalized}, that can be described succinctly by the update
\begin{equation}\label{eq:IHT}
\X^{j+1} = \H_r\left(\X^j + \A^*(\y - \A(\X^j))\right),
\end{equation}
where $\X^0$ is chosen either as the zero vector/matrix or randomly. Here, $\A^*$ denotes the adjoint of the operator $\A$, and the function $\H_r$ is a thresholding operator. In the vector sparse recovery setting, $\H_r$ simply keeps the $r$  largest in magnitude entries of its input and sets the rest to zero, thereby returning the closest $r$-sparse vector to its input. In the matrix case, it returns the closest rank-$r$ matrix to its input. To guarantee recovery via the IHT method, one may consider the restricted isometry property (RIP) \cite{RefWorks:48}, which asks that the operator $\A$ roughly preserves the geometry of sparse/low-rank vectors/matrices:
$$
(1-\delta_r) \|\X\|_F^2 \leq \|\A(\X)\|_2^2 \leq (1+\delta_r)\|\X\|_F^2 \quad\text{for all $r$-sparse/rank-$r$ matrices $\X$},
$$
where $0 < \delta_r < 1$ is a controlled constant that may depend on $r$ and $\| \cdot \|_F$ denotes the Frobenius norm. For example, when the operator $\A$ satisfies the RIP for $3r$-sparse vectors with constant $\delta_{3r} < 1/\sqrt{32}$, after a suitable number of iterations, IHT exactly recovers any $r$-sparse vector $\x$ from the measurements $\y = \A(\x)$. Moreover, the result is robust and shows that when the measurements $\y$ are corrupted by noise $\y = \A\x + \e$ and the vector $\x$ is not exactly sparse but well-approximated by its $r$-sparse representation $\x_r$, that IHT still produces an accurate estimate to $\x$ with error proportional to $\|\e\|_2 + \frac{1}{\sqrt{r}}\|\x-\x_r\|_1 + \|\x-\x_r\|_2$.  See \cite{blumensath2010normalized} for details.

\subsection{Extension to tensor recovery}

Extending IHT from sparse vector recovery to low-rank matrix recovery is somewhat natural. Indeed, a matrix is low-rank if and only if its vector of singular values is sparse. Extensions to the tensor setting, however, yield some non-trivial challenges. Nonetheless, there are many applications that motivate the use of tensors, ranging from video and longitudinal imaging \cite{liu2012tensor,bengua2017efficient} to machine learning \cite{romera2013multilinear,vasilescu2005multilinear} and differential equations \cite{beck2000multiconfiguration,lubich2008quantum}.

We will write a $d$-order tensor as $\X\in\R^{n_1\times n_2\times\ldots\times n_d}$, where $n_i$ denotes the dimension in the $i$th mode. Unlike the matrix case, for order $d\geq 3$ tensors, there is not one unique notion of rank. In fact, many notions of rank along with their corresponding decompositions have now been defined and studied, including Tucker rank and higher order SVD \cite{tucker1966some,de2000multilinear}, CP-rank \cite{carroll1970analysis,harshman1970foundations},  and tubal rank \cite{kilmer2011factorization, zhang2014novel}. We refer the reader to \cite{kolda2009tensor} for a nice review of tensors and these various concepts. Succinctly, the Tucker format relies on \textit{unfoldings} of the tensor whereas the CP format relies on rank-1 tensor factors. For example, an order $d$ tensor $\X$ can be \textit{matricized} in $d$ ways by unfolding it along each of the $d$ modes \cite{kolda2009tensor}. One can then consider a notion of rank for this tensor as a $d$-tuple $(r_1, r_2, \ldots, r_d)$ where $r_i$ is the rank of the $i$th unfolding. Such a notion is attractive since rank is well-defined for matrices.  The CP format on the other hand avoids the need to matricize or unfold the tensor. For the purposes of this work, we will focus on CP-decompositions and CP-rank.

For vectors $\x$ and $\z$ denote by $\x\otimes\z$ their outer product and for any integer $r$, let $ [r] = \{1,2,...,r \}$. Then one can build a tensor in $\R^{n_1\times n_2\times\ldots\times n_d}$ by taking the outer product $\x_1\otimes\x_2\ldots\otimes \x_d$ where $\x_i \in \R^{n_i}$ and $\x_1\otimes\x_2\ldots\otimes \x_d$ is a rank-$1$ tensor. This leads to the notation of a rank-$r$ tensor by considering vectors $\x_{ij}\in\R^{n_j}$ for $i\in[r]$ and $j\in[d]$ and considering the sum of $r$ rank-$1$ tensors:
$$
\X = \sum_{i=1}^r \x_{i1}\otimes\x_{i2}\otimes\ldots\otimes\x_{id}.
$$
When $\X$ can be written via this decomposition, $\X$ is at most rank $r$. The smallest number of rank-$1$ tensors that can be used to express a tensor $\X$ is then defined to be the rank of the tensor, and a decomposition using that number of rank-$1$ tensors is its rank decomposition. Note that often one may also ask that the vectors $\x_{ij}$ have unit norm, and aggregate the magnitude information into constants $\lambda_i$ so that
$$
\X = \sum_{i=1}^r \lambda_i\x_{i1}\otimes\x_{i2}\otimes\ldots\otimes\x_{id}.
$$

Note that there are many differences between matrix rank and tensor rank.
For example, the rank of  a real-valued tensor may be different when considered over $\R$ versus $\C$ (i.e. if one allows the factor vectors above to be complex-valued or restricted to the reals).
Throughout this paper, we will consider real-valued tensors in $\R$, but our analysis extends to the complex case as well. The CP-rank and CP-decompositions can be viewed as natural extensions of the matrix rank and SVD, and are well motivated by applications such as topic modeling, psychometrics, signal processing, linguistics and many others  \cite{carroll1970analysis,harshman1970foundations,anandkumar2014tensor}.

 Unfortunately, not only is rank-minimization of tensors NP-Hard, but even the relaxation to the nuclear norm minimization using any of these notions of rank is also NP-Hard \cite{hillar2013most}.  Therefore, it is even more crucial to consider other types of methods for tensor recovery. Fortunately, many iterative methods have natural extensions to the tensor setting.

Here, we will focus on the extension of the IHT method \eqref{eq:IHT} to the tensor setting, as put forth in \cite{rauhut2017low}. The authors prove accuracy of the tensor variant under a tensor RIP assumption. Likewise, we will consider measurements of the form $\y = \A(\X)$, where $\A : \R^{n_1\times\ldots\times n_d}\rightarrow \R^m$ is a linear operator. The tensor IHT method (TIHT) of \cite{rauhut2017low} is summarized as in Algorithm \ref{alg:TIHT}.

\begin{algorithm}[H]
\caption{Tensor Iterative Hard Thresholding (TIHT)}\label{alg:TIHT}
\begin{algorithmic}[1]
\State\textbf{Input:} operator $ \H_r$, rank $r$, measurements $\y$, number of iterations $T$
\State\textbf{Output:} $\hat{\X}=\X^{T}$.
\State\textbf{Initialize:} $\X^1=\zero$
\For {$j=0,2,\ldots,T-1$}
\State $\W^j = \X^j + \A^*(\y - \A(\X^j))$
\State $\X^{j+1} = \H_r(\W^j)$ 
\EndFor
\end{algorithmic}
\end{algorithm}

Note that the algorithm depends on the ``thresholding'' operator\footnote{We note that this operator need not really be a true ``thresholding'' operator, but use this nomenclature since the method is derived from the classical iterative hard thresholding method.} as input. In the vector case this operator performs simple thresholding, thus the name of the algorithm. In the matrix case it typically performs a low-rank projection via the SVD. In the tensor case, there are several options for this operator. In \cite{rauhut2017low}, the authors ask that $\H_r$ computes an approximation to the closest rank-$r$ tensor to its input. Such an approximation is necessary since computing the best rank-$r$ tensor may be ill-defined or NP-Hard depending on the notion of rank used. For $d$-order tensors, the approximation in \cite{rauhut2017low} is asked to satisfy
$$
\|\X - \H_r(\X)\|_F \leq C\sqrt{d}\|\X - \X_r\|_F,
$$
where $\X_r$ is the best rank-$r$ approximation to $\X$ using various notions of the Tucker rank. We will adapt this operator to the CP-rank and propose a valid approximation for our purposes later. As in \cite{rauhut2017low}, the recovery of low CP-rank tensors by TIHT will rely on a tensor variant of the RIP, defined below in Definition~\ref{cptrip}.

We will utilize the TRIP for an appropriate set $S_{r,R}$ corresponding to normalized tensors with low CP-rank. In \cite{rauhut2017low}, the TRIP is utilized for several other notions of Tucker rank, namely for tensors with low rank higher order SVD (HOSVD) decompositions, hierarchical Tucker (HT) decompositions, and tensor train (TT) decompositions. The authors then prove that various randomly constructed measurement maps $\A$ satisfy those variations of the TRIP with high probability.  Indeed, letting the rank $r$ bound the rank entries $r_i$ of the appropriate Tucker $d$-tuple (see \cite{rauhut2017low} for details), the TRIP is satisfied when the number of measurements $m$ is on the order of $\delta_r^{-2}(r^d + dnr)\log(d)$ and $\delta_r^{-2}(dr^3 + dnr)\log(dr)$ or HOSVD and TT/HT decompositions, respectively. Such random constructions include those obtained by taking tensor inner products of $\X$ with tensors having i.i.d. (sub)-Gaussian entries. Under the TRIP assumption, the main result of \cite{rauhut2017low} shows that TIHT provides recovery of low Tucker rank tensors, as summarized by the following theorem.

\begin{theorem}[\cite{rauhut2017low}]
Let $\A : \R^{n_1\times n_2\times\ldots \times n_d} \rightarrow \R^m$ satisfy the TRIP (for rank$-r$ HOSVD or TT or HT tensors) with $\delta_{3r} \leq \delta < 1$, and run TIHT with noisy measurements $\y = \A\X +\e$.  Assume that the following holds at each iteration of TIHT:
\begin{equation}\label{eq:verify}
\|\W^j - \X^{j+1}\|_F \leq (1+\varepsilon)\|\W^j - \X\|_F.
\end{equation}
Then the estimates produced by TIHT satisfy:
$$
\|\X^{j+1} - \X\|_F \leq c^j\|\X\|_F + C\|\e\|_2,
$$
where $0<c<1$ and $C$ denote constants that may depend on $\delta$.

As a consequence, after $T = C'\log\left(\|\X\|_F / \|\e\|_2\right)$ iterations, the estimate satisfies
$$
\|\X^T - \X\|_F \leq C''\|\e\|_2,
$$
where $C'$ and $C''$ denote constants that may depend on $\delta$.
\end{theorem}

As the authors themselves point out, the challenge with this result is that \eqref{eq:verify} may be challenging to verify. 

\subsection{Contributions and Organization}

{The main contribution of this work is the extension and analysis of TIHT to low CP-rank tensors. Using recent work in low CP-rank tensor approximations, this work provides theoretical guarantees for the recovery of low CP-rank tensors without requiring assumptions on the hard thresholding operation $\mathcal{H}_r$. We also show that tensor measurement maps with properly normalized Gaussian random variables satisfy a CP-rank version of the tensor RIP (TRIP) with high probability. These contributions are then supported by synthetically generated as well as real world experiments on video data.}

The remainder of the paper is organized as follows. Section \ref{sec:main} contains our main results. Theorem \ref{mainthm} proves accurate recovery of TIHT for tensors with low CP-rank, under an appropriate TRIP assumption, without the need to verify an assumption like \eqref{eq:verify}. In Section \ref{sec:maintrip} we prove Theorem \ref{coverthm} showing that measurement maps satisfying our TRIP assumption can be obtained by random constructions. Section \ref{sec:exps} showcases numerical results for real and synthetic tensors, and we conclude in Section \ref{sec:conclude}.

\section{Main Results for TIHT for CP-rank}\label{sec:main}

First, let us formally define the set of tensors for which we will prove accurate recovery using TIHT. Given some $R>0$, let us define the set of tensors:
\begin{equation}\label{SrR}
S_{r,R} := \{ \X \in \R^{n_1\times n_2\times\ldots\times n_d} : \|\X\|_F = 1, \X = \sum_{i=1}^r \x_{i1}\otimes \x_{i2}\otimes\ldots\otimes \x_{id}, x_{ij}\in\R^{n_j}, \|\x_{ij}\|_2 \leq R\}.
\end{equation}
In other words, $S_{r,R}$ is the set of all CP-rank $r$ tensors with bounded factors. Such tensors are not unusual and have been used to provide theoretical guarantees in previous works~\cite{song2019relative,bhaskara2014uniqueness}.

We first define an analog of the tensor RIP (TRIP) for low CP-rank tensors.

\begin{definition}\label{cptrip}
The measurement operator $\A : \R^{n_1\times n_2\times\ldots \times n_d} \rightarrow \R^m$ satisfies the TRIP adapted to $S_{r,R}$ with parameter $\delta_r>0$ when
$$
(1-\delta_r)\|\X\|_F^2 \leq \|\A(\X)\|_2^2 \leq (1+\delta_r)\|\X\|_F^2
$$
for all $\X\in S_{r,R}$, defined in \eqref{SrR}.
\end{definition}

We will utilize the method and result from \cite{song2019relative} that guarantees the following.

\begin{theorem}[\cite{song2019relative}, Theorem 1.2]\label{woodruff}
Let $\W$ be an arbitrary order-$d$ tensor, $\varepsilon,\alpha>0$, and positive integer $r$, and set
$$
\gamma := \min_{\hat{\W} : \text{rank}(\hat{\W})=r} \|\hat{\W} - \W\|_F.
$$
Suppose there is a rank-$r$ tensor $\hat{\W}$ satisfying $\|\hat{\W} - \W\|_F^2 \leq \gamma^2 + 2^{-n^\alpha}$ and whose CP  factors have norms bounded by $2^{n^\alpha}$. Then there is an efficient algorithm that outputs a rank-$r$ tensor estimate $\tilde{\W}$ such that
$$
\|\W - \tilde{\W}\|_F^2 \leq (1+\varepsilon)\gamma^2 + 2^{-n^\alpha}.
$$
We will write this method as $\H_r(\W) = \tilde{\W}$.
\end{theorem}

Our variant of the TIHT method will utilize this result. It can be summarized by the update steps, initialized with $\X^0 = \mathbf{0}$:
\begin{align}
\W^j &= \X^j + \A^*(\y - \A(\X^j)) \label{eq:cpTIHT_GD}\\
\X^{j+1} &= \H_r(\W^j) \quad\text{as in Theorem \ref{woodruff}.\label{eq:cpTIHT_HT}}
\end{align}

 In our context, Theorem \ref{woodruff} means the following.

\begin{corollary}\label{woodcor}
Let $\alpha, \varepsilon >0 $, and let $\X$ be an arbitrary CP-rank $r$ tensor with bounded factors; in particular let $\X\in S_{r, 2^{n^\alpha}}$. Assume $\A$ is a measurement operator satisfying the TRIP with parameter $\delta = \delta_{3r} < \frac{1}{2}2^{-n^\alpha}$.

Assume measurements $\y = \A(\X) + \z$ with bounded noise $\|\z\|_2 \leq \frac{2^{-0.5n^\alpha}}{2\|\A\|_{2\rightarrow 2}}$. Using the notation in~\eqref{eq:cpTIHT_GD} and~\eqref{eq:cpTIHT_HT}, we have
$$
\|\W^0 - \X^{1}\|_F^2 \leq (1+\varepsilon)\|\W^0 - \X\|_F^2 + 2^{-n^\alpha},
$$
and $\X^1$ is of CP-rank $r$ with factors bounded by $2^{n^\alpha}$.

In addition, if $\|\X^{j} - \X\|_F \leq \|\X^{0} - \X\|_F$, then we have that
$$
\|\W^{j} - \X^{j+1}\|_F^2 \leq (1+\varepsilon)\|\W^j - \X\|_F^2 + 2^{-n^\alpha},
$$
and $\X^{j+1}$ is of CP-rank $r$ with factors bounded by $2^{n^\alpha}$.
\end{corollary}
\begin{proof}
To apply Theorem \ref{woodruff}, we will verify there is a rank-$r$ tensor $\hat{\W}$ with bounded factors that satisfies $\|\hat{\W} - \W^0\|_F^2 \leq  2^{-n^\alpha} \leq \gamma^2 + 2^{-n^\alpha}$. Our choice for $\hat{\W}$ is precisely the tensor $\X$. Indeed, we have
\begin{align*}
\|\W^0 - \X\|_F &= \|\A^*\A(\X) + \A^*\z - \X\|_F\\
&\leq \|(\A^*\A - \Id)(\X)\|_F + \|\A^*\z\|_F\\
&\leq \delta\|\X\|_F +  \|\A\|_{2\rightarrow 2}\|\z\|_2\\
&\leq \frac{1}{2}2^{-n^\alpha} + \frac{1}{2}2^{-0.5n^\alpha}.
\end{align*}
Thus, $\|\W^0 - \X\|_F^2 \leq (\frac{1}{2}2^{-n^\alpha} + \frac{1}{2}2^{-0.5n^\alpha})^2 \leq 2^{-n^\alpha}$, so by Theorem \ref{woodruff}, the output $\X^1$ satisfies
$$
\|\W^0 - \X^{1}\|_F^2 \leq (1+\varepsilon)\min_{\hat{\W} : \text{rank}(\hat{\W})=r} \|\hat{\W} - \W^1\|_F^2 + 2^{-n^\alpha} \leq (1+\varepsilon)\|\W^0 - \X\|_F^2 + 2^{-n^\alpha}.
$$
To prove the second part, we proceed in the same way. Namely, we have
\begin{align*}
\|\W^j - \X\|_F &= \|\X^j + \A^*\A(\X) + \A^*\z - \A^*\A\X^j - \X\|_F\\
&\leq \|(\A^*\A - \Id)(\X - \X^j)\|_F + \|\A^*\z\|_F\\
&\leq \delta\|\X - \X^j\|_F +  \|\A\|_{2\rightarrow 2}\|\z\|_2\\
&\leq  \delta\|\X - \X^0\|_F +  \|\A\|_{2\rightarrow 2}\|\z\|_2\\
&= \delta\|\X\|_F +  \|\A\|_{2\rightarrow 2}\|\z\|_2\\
&\leq \frac{1}{2}2^{-n^\alpha} + \frac{1}{2}2^{-0.5n^\alpha}.
\end{align*}
Thus, $\|\W^j - \X\|_F^2 \leq (\frac{1}{2}2^{-n^\alpha} + \frac{1}{2}2^{-0.5n^\alpha})^2 \leq 2^{-n^\alpha}$, so by Theorem \ref{woodruff}, the output $\X^{j+1}$ satisfies
$$
\|\W^j - \X^{j+1}\|_F^2 \leq (1+\varepsilon)\min_{\hat{\W} : \text{rank}(\hat{\W})=r} \|\hat{\W} - \W^j\|_F^2 + 2^{-n^\alpha} \leq (1+\varepsilon)\|\W^j - \X\|_F^2 + 2^{-n^\alpha}.
$$
\end{proof}

Our goal will be to prove that the TIHT variant described in~\eqref{eq:cpTIHT_GD}-\eqref{eq:cpTIHT_HT} provides accurate recovery of tensors in $S_{r,R}$ \eqref{SrR}, provided that the measurement operator $\A$ satisfies the CP-rank analog of the TRIP.

We now proceed with our main theorems.

\begin{theorem}[TIHT with bounded low CP-rank]\label{mainthm}
Let $\X\in S_{r,2^{n^\alpha}}$.
Consider the TIHT method described in~\eqref{eq:cpTIHT_GD}-\eqref{eq:cpTIHT_HT}, assume $\A$ satisfies the TRIP with parameter $\delta=\delta_{3r} \leq \frac{1}{2}2^{-n^\alpha}$ as in Definition~\ref{cptrip}, and run TIHT with noisy measurements $\y = \A(\X) + \z$ where the noise is bounded $\|\z\|_2 \leq \frac{2^{-n^\alpha}}{2\|\A\|_{2\rightarrow 2}}$. Then TIHT has iterates that satisfy
\begin{equation}
\|\X^{j+1} - \X \|_F \leq (2 \delta)^j \|\X^{0} - \X \|_F + \frac{2 \sqrt{1+\delta}}{1 - 2 \delta}\|\z \|_2 + \frac{(1+\epsilon) 2^{-0.5n^\alpha}}{1 - 2 \delta}.
\end{equation}
As a consequence, recovery error on the order of the upper bound of the noise, $2^{-n^\alpha}$, is achieved after roughly $\lceil \log_{1/2\delta}(\|\X^0 - \X\|_F / \|\z\|_2) \rceil $ iterations.
\end{theorem}

\begin{proof}
The proof follows that of Theorem 1 in \cite{rauhut2017low} with some crucial modifications. In particular, instead of requiring assumption (31) in the proof the aforementioned theorem, we have Corollary~\ref{woodcor}. As a direct result, we also have a nice upper bound on $\|\W^j - \X \|_F$ whereas Theorem 1 in \cite{rauhut2017low} requires additional computation for upper bounding this term.

Starting from Corollary~\ref{woodcor}, we have that
\begin{equation*}
2^{-n^\alpha} + (1+\epsilon) \|\W^j - \X \|^2_F \geq \|\W^{j} - \X^{j+1} \|^2_F.
\end{equation*}
Adding and subtracting $\X$ to the right hand side, rearranging terms, and substituting the value of $\W^j$ from~\eqref{eq:cpTIHT_GD}, we can write:
\begin{align*}
\|\X^{j+1} - \X \|^2_F &\leq 2 \langle \X^j - \X, \X^{j+1} - \X \rangle  - 2 \langle \A ( \X^j - \X), \A(\X^{j+1} - \X) \rangle\\
&\quad\quad +2 \langle \z, \A(\X^{j+1} - \X) \rangle + (2\epsilon + \epsilon^2) \| \W^j - \X \|_F^2 + 2^{-n^\alpha}.
\end{align*}
Using the fact that  $rank(\X^{j+1} - \X) \leq 2r < 3r$, we invoke TRIP and the Cauchy-Schwarz inequality on the third term and the upper bound $\| \W^j - \X\|_F^2 \leq 2^{-n^\alpha}$ shown in Corollary~\ref{woodcor} on the fourth term in the summation to obtain
\begin{align*}
\|\X^{j+1} - \X \|^2_F &\leq 2 \langle \X^j - \X, \X^{j+1} - \X \rangle  - 2 \langle \A ( \X^j - \X), \A(\X^{j+1} - \X) \rangle\\ &\quad\quad +2 \sqrt{1+\delta_{3r} }  \| \X^{j+1} - \X \|_F\|\z\|_2  + (1+\epsilon)^2 2^{-n^\alpha}.
\end{align*}
Let $\mathcal{Q}^j:\mathbb{R}^{n_1 \times n_2 \times \dots \times n_d } \rightarrow \mathbb{R}^{U^j}$ by the orthogonal projection operator into the subspace spanned by $\X^{j+1}$, $\X^j$, and $\X$. Additionally denote $\A_\mathcal{Q}^j(\Z):=\A (\mathcal{Q}^j(\Z))$ for all $\Z \in \mathbb{R}^{n_1\times n_2 \times \dots n_d}$. Using this notation, we can rewrite the above inequality as:
\begin{align*}
\|\X^{j+1} - \X \|^2_F &\leq 2 \langle \X^j - \X, \X^{j+1} - \X \rangle  - 2 \langle \A_\mathcal{Q}^j ( \X^j - \X), \A_\mathcal{Q}^j (\X^{j+1} - \X) \rangle\\ &\quad\quad +2 \sqrt{1+\delta_{3r} }  \| \X^{j+1} - \X \|_F\|\z\|_2  + (1+\epsilon)^2 2^{-n^\alpha} \\
&\leq 2 \| \Id - {\A_\mathcal{Q}^j}^*\A_\mathcal{Q}^j \|_{2 \rightarrow 2} \|\X^{j+1} - \X \|_F \|\X^{j} - \X \|_F \\ &\quad\quad + 2 \sqrt{1+\delta_{3r} }  \| \X^{j+1} - \X \|_F\|\z\|_2  + (1+\epsilon)^2 2^{-n^\alpha},
\end{align*}
where the final inequality uses simplification to combine the first two terms and the Cauchy-Schwarz inequality.

Let $\beta$, $\gamma \in [0,1]$ such that $\beta + \gamma = 1$ then we can write:
\begin{align*}
(1-\beta-\gamma)\|\X^{j+1} - \X \|^2_F & \leq 2 \| \Id - {\A_\mathcal{Q}^j}^*\A_\mathcal{Q}^j \|_{2 \rightarrow 2} \|\X^{j+1} - \X \|_F \|\X^{j} - \X \|_F \\
\beta \|\X^{j+1} - \X \|^2_F &\leq 2 \sqrt{1+\delta_{3r} }  \| \X^{j+1} - \X \|_F\|\z\|_2 \\
\gamma \|\X^{j+1} - \X \|^2_F &\leq (1+\epsilon)^2 2^{-n^\alpha}.
\end{align*}
Dividing the first two inequalities by $\| \X^{j+1} - \X\|_F$, square rooting both sides of the last inequality, and taking the sum of all three inequalities results in:
\begin{equation}
\|\X^{j+1} - \X \|_F \leq f(\beta) \left( 2 \| \Id - {\A_\mathcal{Q}^j}^*\A_\mathcal{Q}^j \|_{2 \rightarrow 2} \|\X^{j} - \X \|_F + 2 \sqrt{1+\delta_{3r} }\|\z \|_2 + (1+\epsilon) 2^{-0.5n^\alpha}\right),
\end{equation}
where $f(\beta) = \left(1 - \beta + \sqrt{\beta} \right)$. Noting that $0 \leq f(\beta) \leq 1$ for $\beta \in [0,1]$, we can drop the $f(\beta)$ term proceeding inequalities. With the bound from the proof of \cite{rauhut2017low}[Theorem 1], we have that $\| \Id - {\A_\mathcal{Q}^j}^*\A_\mathcal{Q}^j \|_{2 \rightarrow 2} \leq \delta_{3r}$ leading to the inequality:
\begin{equation*}
\|\X^{j+1} - \X \|_F \leq 2 \delta_{3r} \|\X^{j} - \X \|_F + 2 \sqrt{1+\delta_{3r}}\|\z \|_2 + (1+\epsilon) 2^{-0.5n^\alpha}.
\end{equation*}
Finally, iterating the upper bound leads to the desired result:
\begin{equation*}
\|\X^{j+1} - \X \|_F \leq (2 \delta_{3r})^j \|\X^{0} - \X \|_F + \frac{2 \sqrt{1+\delta_{3r}}}{1 - 2 \delta_{3r}} \|\z \|_2+ \frac{(1+\epsilon) 2^{-0.5n^\alpha}}{1 - 2 \delta_{3r}}.
\end{equation*}
\end{proof}

Note that the TIHT method and main theorem have been modified in two crucial ways. First, the ``thresholding'' operator $\H_r$ has been replaced by the output guaranteed by Theorem \ref{woodruff}. This allows us to obtain an efficient thresholding step at the price of outputting a tensor of slightly higher rank. This higher rank output in turn requires a stricter assumption on the operator $\A$, namely that the TRIP is satisfied with that higher rank. However, we typically assume $d$ is small and bounded, so the increase is not severe. That leads to the second modification, which is that the TRIP is defined for low CP-rank matrices rather than other types of low rankness. Thus, what remains to be proved is for what types of measurement operators $\A$ this TRIP holds.

\subsection{Measurement maps satisfying our TRIP}\label{sec:maintrip}

The following theorem shows that Gaussian measurement maps satisfy the desired TRIP with high probability.

\begin{theorem}\label{coverthm}
Let $\A : \R^{n_1\times n_2\times\ldots \times n_d} \rightarrow \R^m$ be represented by a tensor in $\R^{n_1\times n_2\times\ldots \times n_d\times m}$ whose entries are properly normalized, i.i.d. Gaussian random variables.\footnote{This can easily be extended to mean zero $L$-subgaussian random variables, as in \cite{rauhut2017low}, where the constant $C$ in \eqref{eq:mtrip} depends on the subgaussian parameter $L$.}
Then $\A$ satisfies the TRIP as in Definition \ref{cptrip} with parameter $\delta_r$ as long as
\begin{equation}
m \geq C\delta^{-2}\cdot \max\left\{\log(\varepsilon^{-1}), ~r\log(drR^d)\sum_{i=1}^d n_i \right\}.
\label{eq:mtrip}
\end{equation}
\end{theorem}

\begin{proof}

We proceed again as in \cite{rauhut2017low}, with the main change being the construction of the covering of the set $S_{r,R}$. To that end, we first obtain a bound on the covering number of this set, defined to the minimal cardinality of an $\epsilon$-net $\mathcal{X}$ such that for any point $\X\in S_{r,R}$, there is a point $\hat{\X}\in\mathcal{X}$ such that $\|\X - \hat{\X}\|_F \leq \epsilon$. We denote this covering number as $\mathcal{N}(S_{r,R},\epsilon)$.  See e.g. \cite{rogers1964packing} for more details on covering numbers.

\begin{lemma}\label{coverlem}
For any $\epsilon > 0$, the covering number of $S_{r,R}$ satisfies
$$
\mathcal{N}(S_{r,R},\epsilon) \leq \left(\frac{3drR^d}{\epsilon}\right)^{r\sum_{i=1}^d n_i}.
$$
\end{lemma}
\begin{proof}
For simplicity, we will prove the result for $d=3$ and $n_1 = n_2 = n_3 := n$, and the general case follows similarly. Denote by $B_2^n$ the unit ball in $\R^n$, namely $B_2^n = \{\x\in\R^n : \|\x\|_2 \leq 1\}$.
First, we begin by obtaining an $\epsilon_1$-net $\mathcal{X}_1$ of the ball of radius $R$, namely $RB_2^n$ ($\epsilon_1$ will be chosen later). Classical results~\cite{vershynin2010introduction}
utilizing volumetric estimates show that such an $\epsilon_1$-net exists with cardinality at most $(3R/\epsilon_1)^n$.  Create the set of tensors
$$\mathcal{X}_2 := \left\{\sum_{i=1}^r \x_{i1}'\otimes \x_{i2}'\otimes \x_{i3}' : \x_{ij}'\in\mathcal{X}_1 \right\} \subset\R^{n^3},$$
 and note that that its cardinality is at most
\begin{equation}\label{cover}
 |\mathcal{X}_2| \leq |\mathcal{X}_1|^{3r} \leq (3R/\epsilon_1)^{3nr}.
\end{equation}

Now let $\X\in S_{r,R}$ be given. Thus, $\X$ can be written as $\X = \sum_{i=1}^r \x_{i1}\otimes \x_{i2}\otimes \x_{i3}$ for $\x_{ij}\in RB_2^n$. For each $1\leq i \leq r$ and $1\leq j \leq 3$, choose $\x_{ij}'\in \mathcal{X}_1$ so that $\|\x_{ij} - \x_{ij}'\|_2 \leq \epsilon_1$, and note that $\X' := \sum_{i=1}^r \x_{i1}'\otimes \x_{i2}'\otimes \x_{i3}' \in \mathcal{X}_2$. First, we want to bound a single term $\|\x_{i1}\otimes \x_{i2}\otimes \x_{i3} - \x_{i1}'\otimes \x_{i2}'\otimes \x_{i3}'\|_F$. For notational convenience, let us drop the subscript $i$, and write $\x_{j,a}$ to denote the $a$th entry of the vector $\x_j := \x_{ij}$. In addition, let us define the gradually modified tensors $\Y = \x_1\otimes \x_2\otimes \x_3$,  $\Y' = \x_1'\otimes \x_2\otimes \x_3$, $\Y'' = \x_1'\otimes \x_2'\otimes \x_3$,  and $\Y''' = \x_1'\otimes \x_2'\otimes \x_3'$. Our goal is thus to bound $\|\Y - \Y'''\|_F$. To that end, we have
\begin{align*}
\|\Y - \Y'\|_F &=
\|\x_1\otimes \x_2\otimes \x_3 - \x_1'\otimes \x_2\otimes \x_3\|_F  \\
&= \left( \sum_{a,b,c=1}^n (\x_{1,a}\x_{2,b}\x_{3,c} - \x_{1,a}'\x_{2,b}\x_{3,c})^2\right)^{1/2}\\
&= \left( \sum_{b,c=1}^n\x_{2,b}^2\x_{3,c}^2\sum_{a=1}^n (\x_{1,a} - \x_{1,a}')^2\right)^{1/2}\\
&=  \left( \sum_{b,c=1}^n\x_{2,b}^2\x_{3,c}^2\|\x_{1} - \x_{1}'\|^2\right)^{1/2}\\
&\leq \epsilon_1\left( \sum_{b,c=1}^n\x_{2,b}^2\x_{3,c}^2\right)^{1/2}\\
&= \epsilon_1 \|\x_2\|_2 \|\x_3\|_2\\
&\leq \epsilon_1 R^2.
\end{align*}

Similarly, we bound $\|\Y' - \Y''\|_F \leq \epsilon_1 R^2$ and $\|\Y'' - \Y'''\|_F \leq \epsilon_1 R^2$. Thus,
\begin{align*}
\|\Y - \Y'''\|_F &\leq \|\Y -\Y'\|_F + \|\Y' - \Y''\|_F + \|\Y'' - \Y'''\|_F\\
&\leq 3\epsilon_1 R^2
\end{align*}

Finally,
\begin{align*}
\|\X - \X'\|_F  &= \|\sum_{i=1}^r \x_{i1}\otimes \x_{i2}\otimes\ldots\otimes \x_{id} - \sum_{i=1}^r \x_{i1}'\otimes \x_{i2}'\otimes\ldots\otimes \x_{id}'\|_F\\
&\leq \sum_{i=1}^r \|\x_{i1}\otimes \x_{i2}\otimes\ldots\otimes \x_{id} - \x_{i1}'\otimes \x_{i2}'\otimes\ldots\otimes \x_{id}'\|_F\\
&\leq 3r\epsilon_1 R^2.
\end{align*}

Thus, $\mathcal{X}_2$ is a $(3r\epsilon_1 R^2)$-net for $S_{r,R}$. Since our goal is to obtain an $\epsilon$-net, we choose $\epsilon_1 = \epsilon / (3rR^2)$. By \eqref{cover}, this yields a covering number of
$$
|\mathcal{X}_2| \leq \left(\frac{9rR^3}{\epsilon}\right)^{3nr}.
$$
Note that in the case of arbitrary $d$-order tensors, the same proof yields
$$
|\mathcal{X}_2| \leq \left(\frac{3drR^d}{\epsilon}\right)^{r\sum_{i=1}^d n_i}.
$$
\end{proof}

With Lemma \ref{coverlem} in tow, we may now use more classical results from concentration inequalities (see e.g. the proof of Theorem 2 of \cite{rauhut2017low} for a complete proof in the tensor case) that show the number of measurements for a sub-Gaussian operator to satisfy the RIP with parameter $\delta$ over a space $S$ scales like $\log\mathcal{N}(S, \epsilon)$. This completes our proof.
\end{proof}

Since Corollary \ref{woodcor} and Theorem \ref{mainthm} require the TRIP with parameter $\delta_{3r} < \frac{1}{2}2^{-n^\alpha}$, we will apply Theorem \ref{coverthm} with $R = \frac{1}{2}2^{-n^\alpha}$ (and $r$ replaced by $3r$). Theorem \ref{coverthm} then immediately yields the following. Note there is of course a tradeoff in choosing $\alpha$ small or large. Larger $\alpha$ allows for larger noise tolerance and a weaker TRIP requirement, but incurs additional runtime cost in utilizing Corollary \ref{woodcor}.

\begin{corollary}
A Gaussian measurement operator $\A$ will satisfy the TRIP assumptions of Theorem \ref{mainthm} and Corollary \ref{woodcor} so long as
$$
m \geq C'\delta^{-2}\cdot \max\left(\log(\varepsilon^{-1}, r\log(dr2^{-dn^\alpha})\sum_{i=1}^d n_i \right).
$$
\end{corollary}

\section{Numerical Results}\label{sec:exps}
Here we present some synthetic and real experiments that showcase the performance of TIHT when applied to low CP-rank tensors. All experiments are done on order-3 tensors, so $d=3$. We utilize Gaussian measurements as motivated by Theorem \ref{coverthm}. To be precise, for a tensor $\X\in\R^{n_1\times n_2\times n_3}$, we construct for specified number of measurements $m$ a matrix $\A\in\R^{n_1 n_2 n_3\times m}$ whose entries are i.i.d. Gaussian with mean zero and variance $1/m$.  We then compute the measurements $\y = \A\X$ by applying the matrix $\A$ to the vectorized tensor $\X$. Note there are of course other natural ways of applying such Gaussian operators and we do not seek optimality in terms of computation here. In addition, the ``thresholding'' step $\mathcal{H}_r$ of \eqref{eq:cpTIHT_HT} is performed using the \texttt{cpd} function from the Tensorlab package \cite{vervliet2016tensorlab}, which we note may be implemented differently than the method guaranteed by Theorem \ref{woodruff}.  We measure the relative recovery error as $\|\X-{\X^j}\|_F / \|\X\|_F$ for a given tensor $\X$. We refer to the measurement rate as the percentage of the total number of pixels; a rate of $C\%$ corresponds to number of measurements $m = \frac{C}{100}n_1 n_2 n_3$. The rank reported is the rank used in the TIHT method \eqref{eq:cpTIHT_HT}. We present recovery results for this model here.

\subsection{Synthetic data}
For our first set of experiments, we create synthetic low-rank tensors and test the performance of TIHT with varying parameters $r$, $m$ and $\|\z\|_2$, corresponding to the tensor rank, number of measurements, and the noise level, respectively. All tensors are created with i.i.d. standard normal entries, and the noise is also Gaussian, normalized as described. Figure \ref{fig:synth} displays results showcasing the effect of varying the number of measurements  as well as running TIHT on tensors of varying ranks. The left and center plots of this figure show that as expected, lower numbers of measurements lead to slower convergence rates in both the noiseless and noisy case. In the noisy case (center), we also see different levels of the so-called ``convergence horizon.'' We see the same effect in the right plot, showing that lower rank tensors exhibit faster convergence, again as expected. We repeat these experiments but for the tensor completion setting, where instead of taking Gaussian measurements we simply observe randomly chosen entries of the tensor. The results are displayed in Figure \ref{fig:synth2}, where we observe similar behavior.

\begin{figure}[h!]
\includegraphics[width=2.1in]{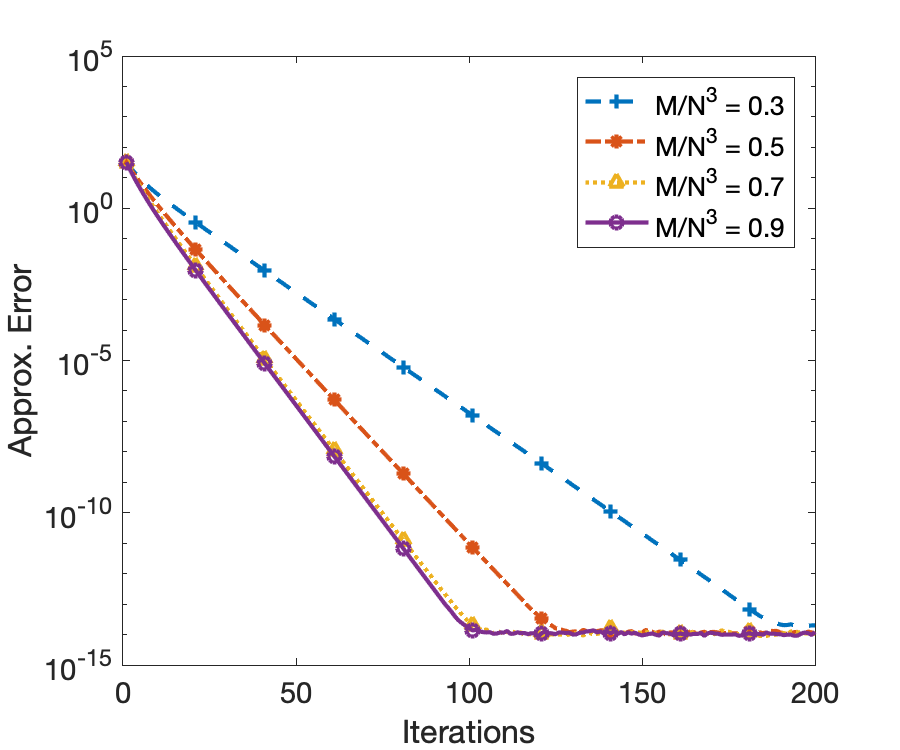}$\;$\includegraphics[width=2.1in]{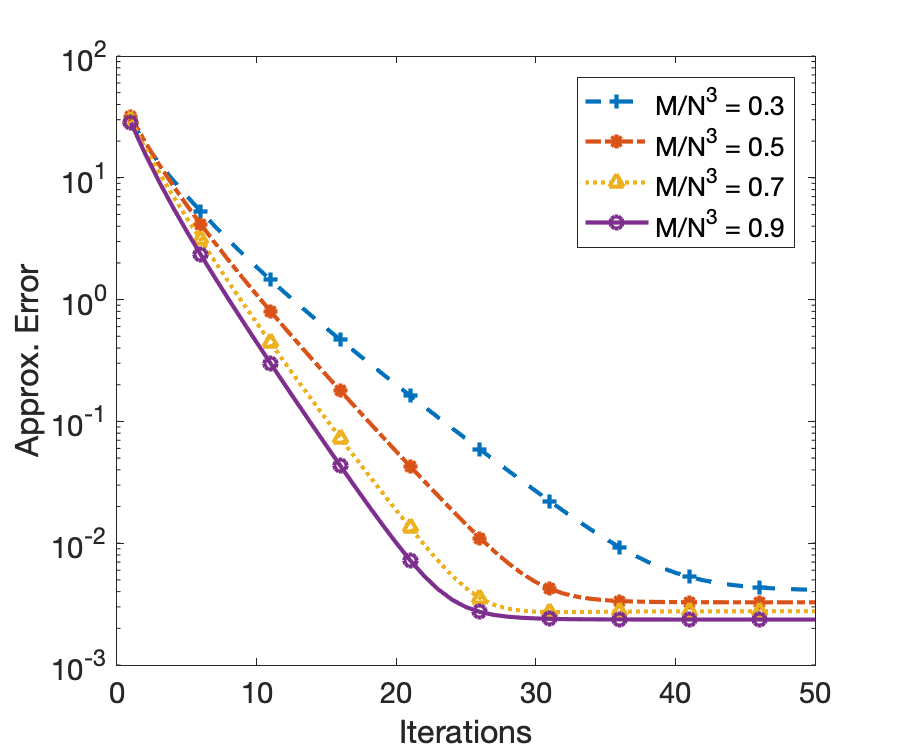}$\;$\includegraphics[width=2.1in]{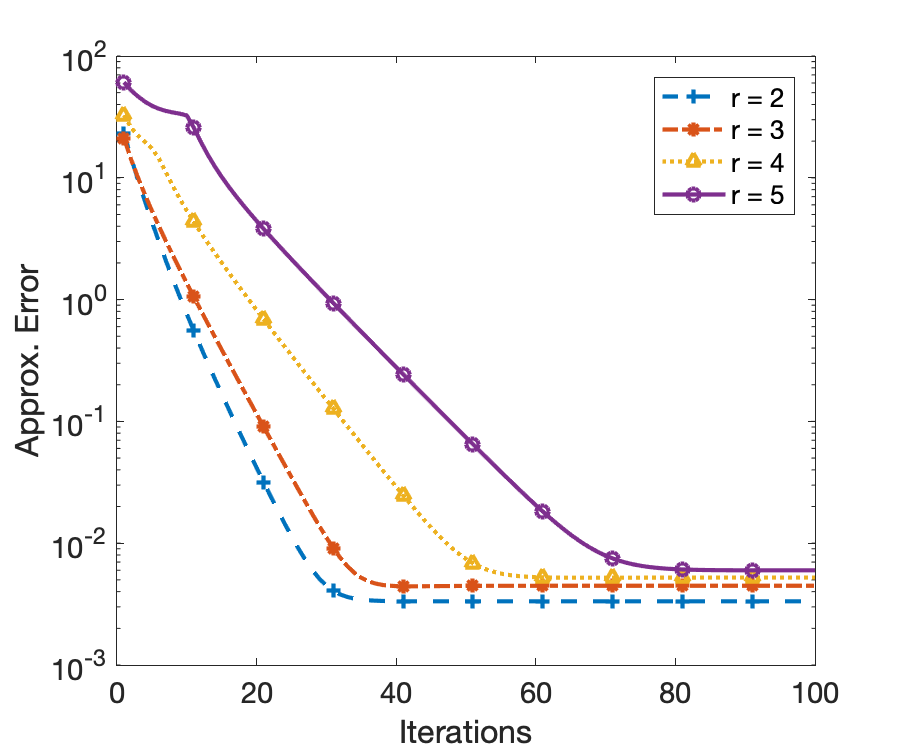}
\caption{Gaussian measurements. Relative recovery error as a function of iteration. Left: Various measurement rates with $r=2$, $n_1=n_2=n_3=10$, no noise $\z=0$. Center: Various measurement rates with $r=2$, $n_1=n_2=n_3=10$,  noise level $\|\z\|_2=0.01$. Right: Various ranks $r$ with $n_1=n_2=n_3=10$,  noise level $\|\z\|_2=0.01$.}\label{fig:synth}
\end{figure}

\begin{figure}[h!]
\includegraphics[width=2.1in]{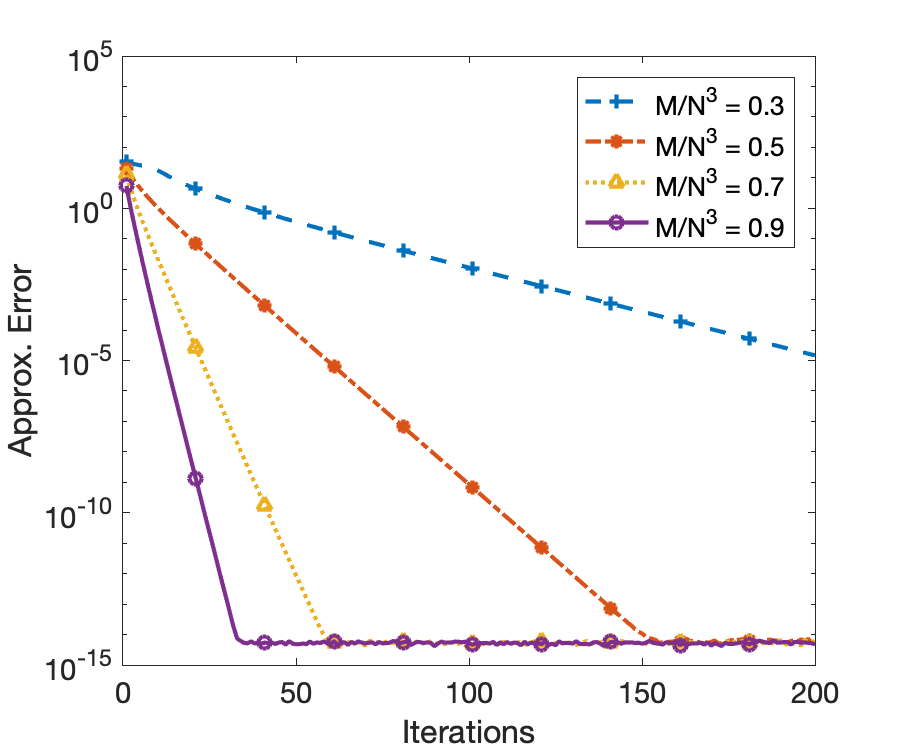}$\;$\includegraphics[width=2.1in]{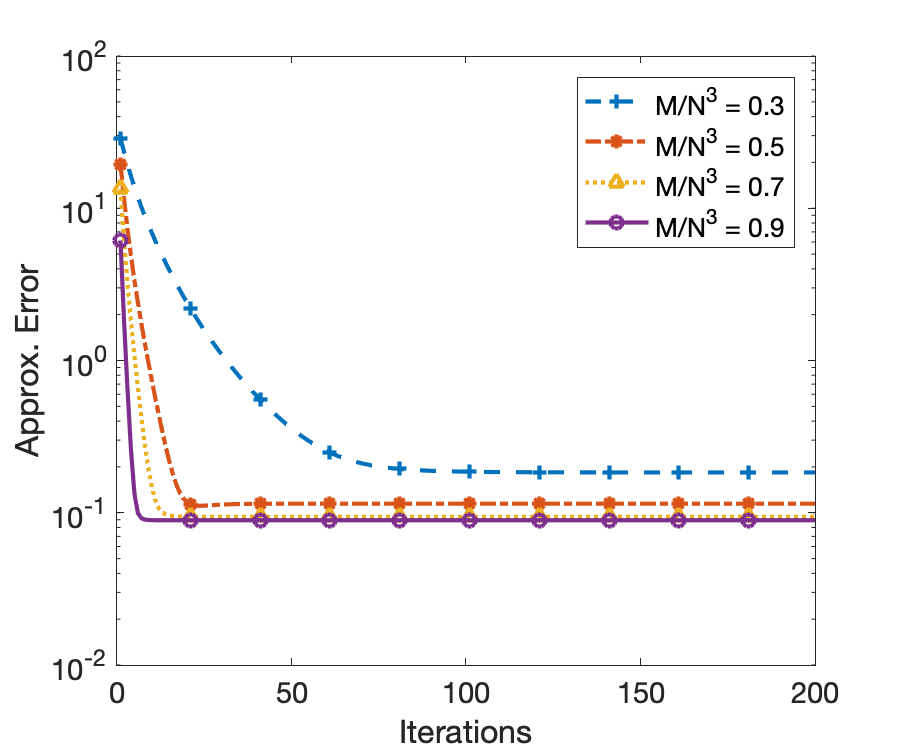}$\;$\includegraphics[width=2.1in]{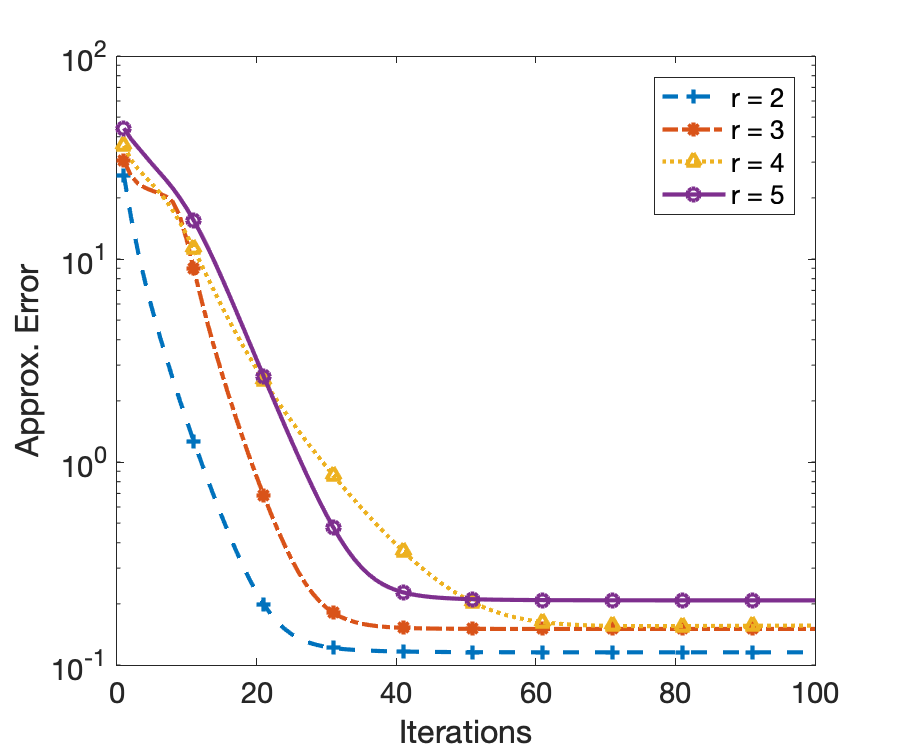}
\caption{Tensor completion. Relative recovery error as a function of iteration. Left: Various measurement rates with $r=2$, $n_1=n_2=n_3=10$, no noise $\z=0$. Center: Various measurement rates with $r=2$, $n_1=n_2=n_3=10$,  noise level $\|\z\|_2=0.01$. Right: Various ranks $r$ with $n_1=n_2=n_3=10$,  noise level $\|\z\|_2=0.01$.}\label{fig:synth2}
\end{figure}

\subsection{Real data}
There is an abundance of real applications that involve tensors. Here, we consider only two, that are appealing for presentation reasons, and for the sole purpose of verifying our theoretical findings. 
Namely, we will consider tensors arising from color images and from grayscale video. For $n_1\times n_2$ color (RGB) images, we view the image as a tensor in $\R^{n_1\times n_2\times 3}$ where the third mode represents the three color channels. Similarly, we use a $n_1\times n_2\times n_3$ tensor representation for a video with $n_3$ frames, each of size $n_1\times n_2$. These experiments are meant to showcase that TIHT can indeed be reasonably applied to such real applications.

Figure \ref{fig:kopen} shows the relative recovery error as a function of TIHT iterations for the RGB images shown, which are already low-rank (by construction).  Figure \ref{fig:truekopen} shows the same result for a real image which is not made to be low-rank. We compute its approximate distance to its nearest rank-15 tensor using the TensorLab \texttt{cpd} function to be $0.1$. Note that we use small image sizes for sake of computation (the image is actually a patch out of a larger image), which causes the lack of sharpness in the visualization of the original images as displayed. We see in Figure~\ref{fig:truekopen} that the relative error reaches around the optimal, namely the relative distance to its low-rank representation. Thus, this error can be viewed as reaching the noise floor. Figure \ref{fig:candle} shows the same information for the videos whose frames are displayed, and we see again that the error decays until approximately the noise floor. Note that for computational reasons we consider only small image sizes, since the measurement maps themselves grow quite large quickly. Of course, other types of maps may reduce this problem, although that is not the focus of this paper.

\begin{figure}[h!]
\includegraphics[width=2.1in]{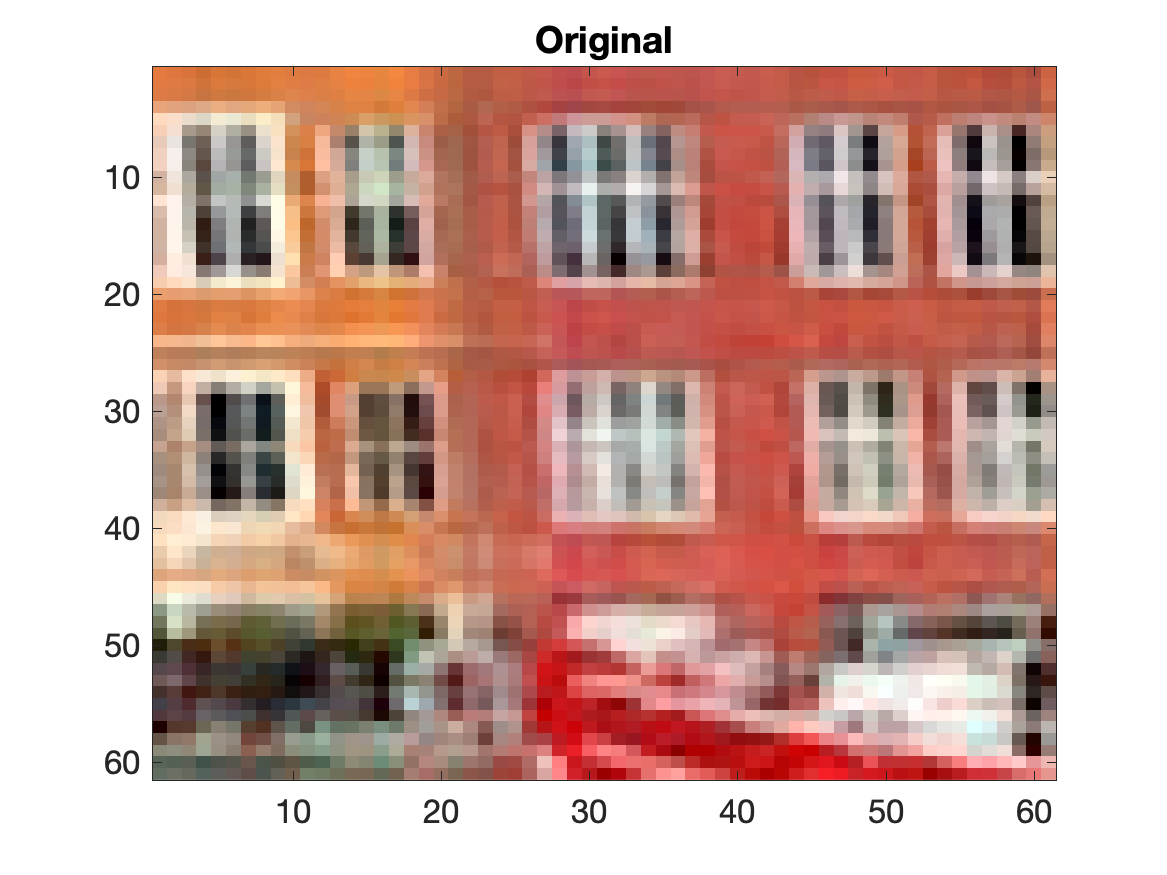}$\;$\includegraphics[width=2.1in]{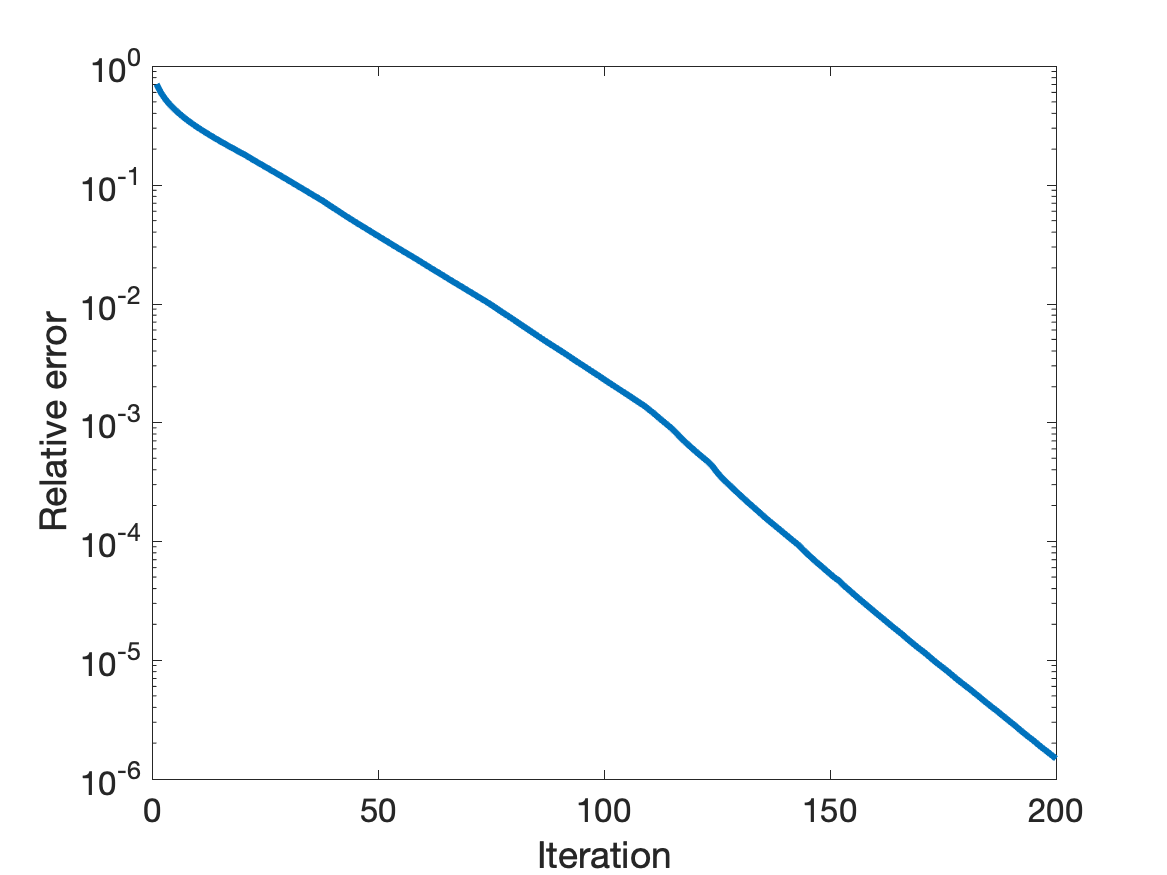}$\;$\includegraphics[width=2.1in]{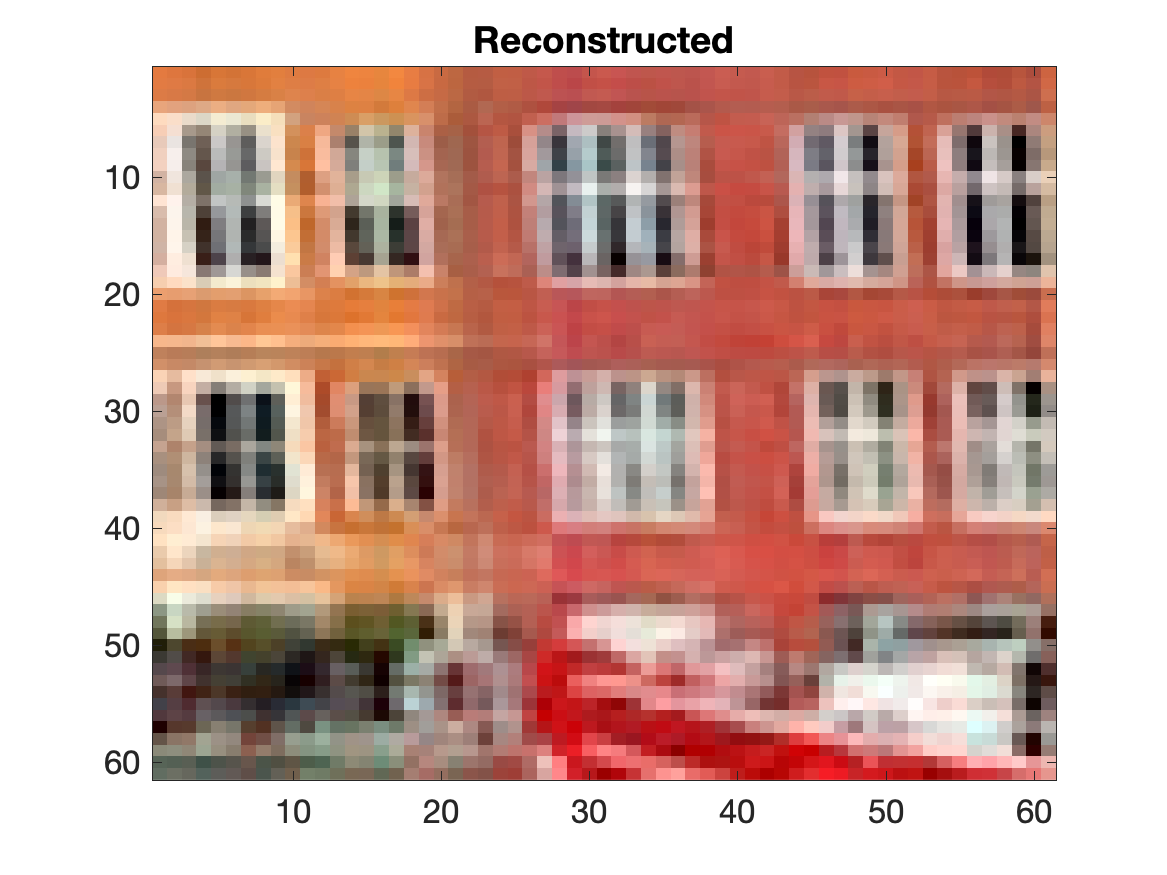}
\caption{Color image of size $60\times 60\times 3$, $60\%$ measurement rate, rank $r=15$. Left: Low rank original image. Center: Relative reconstruction error per iteration. Right: Reconstructed image after $200$ iterations. }\label{fig:kopen}
\end{figure}

\begin{figure}[h!]
\includegraphics[width=2.1in]{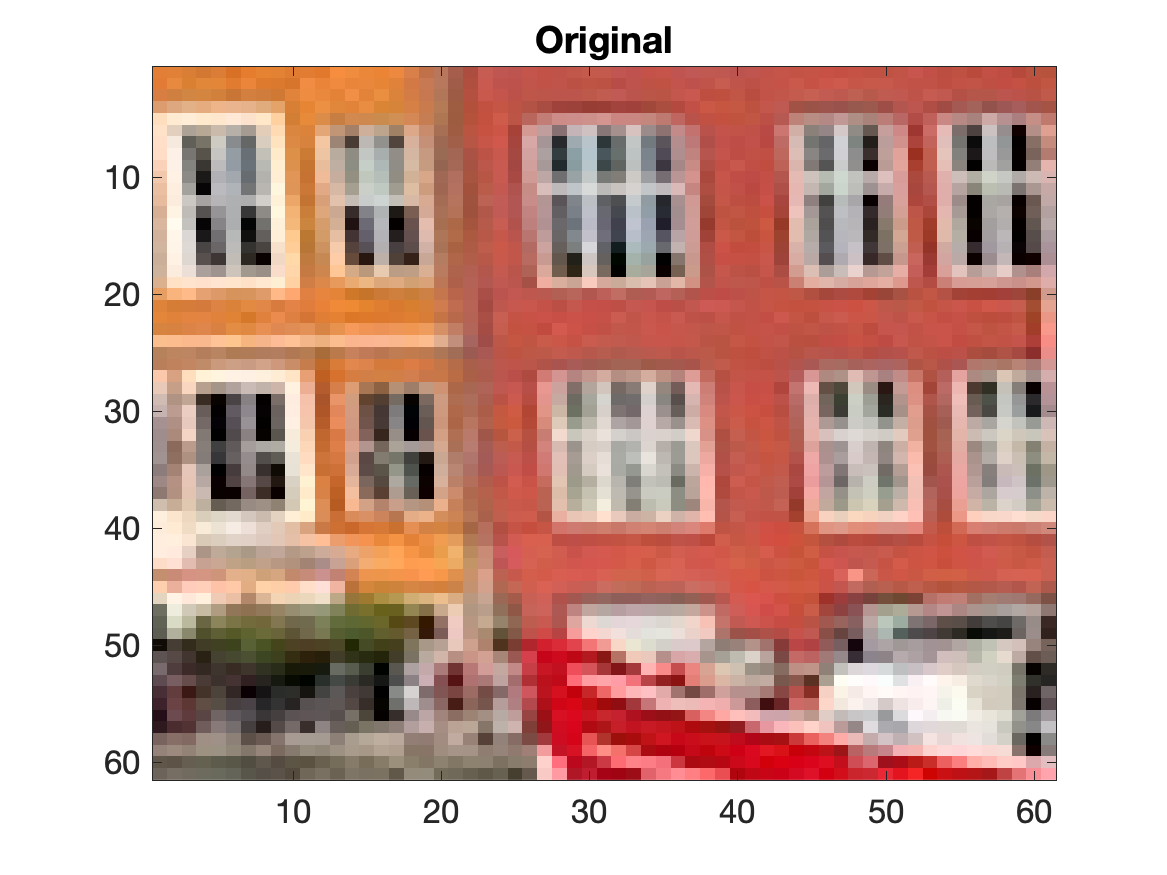}$\;$\includegraphics[width=2.1in]{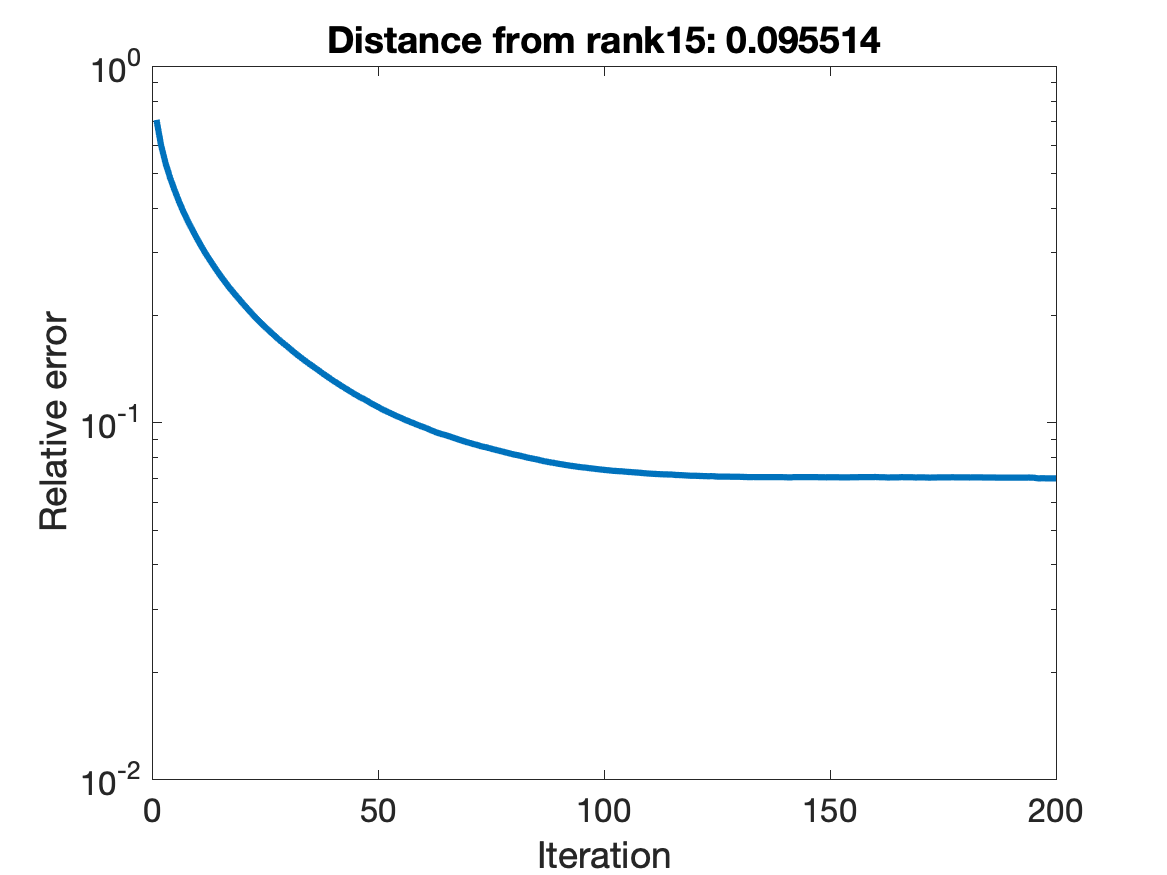}$\;$\includegraphics[width=2.1in]{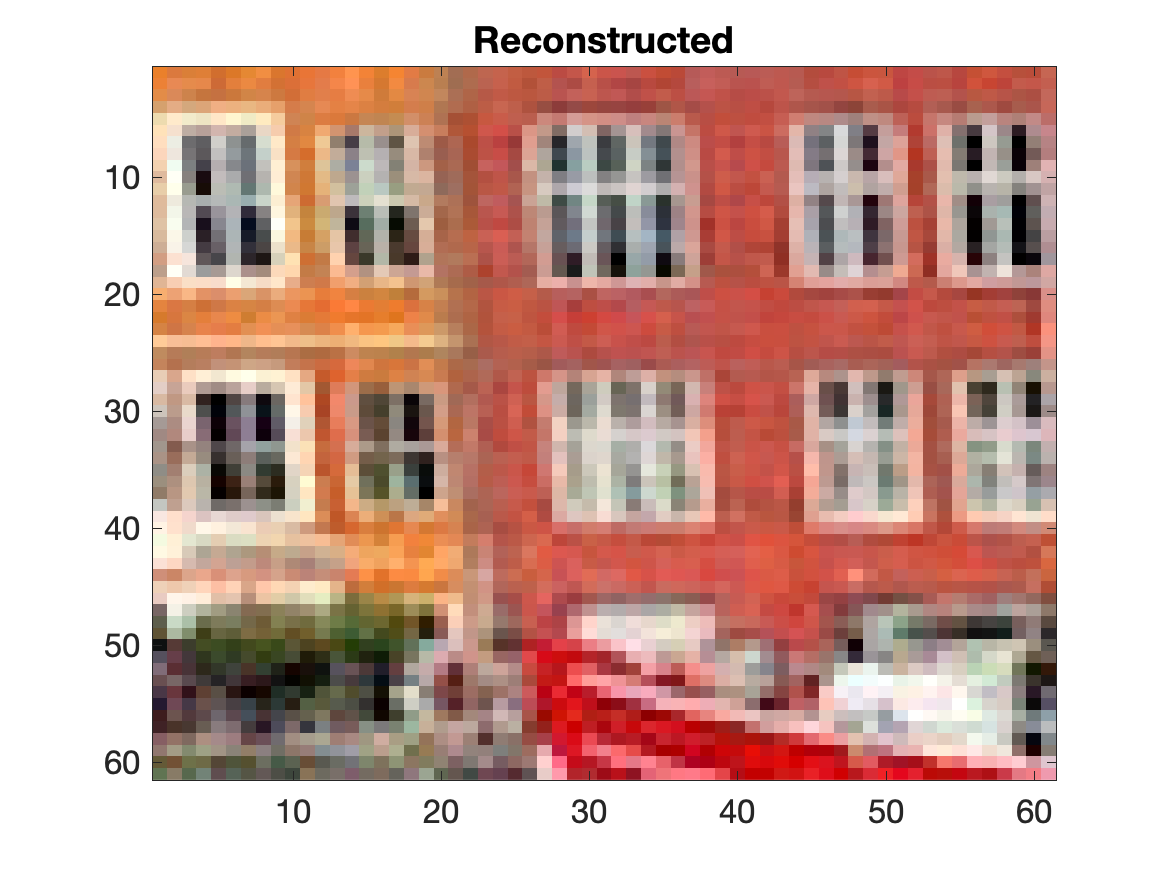}
\caption{Color image of size $60\times 60\times 3$, $60\%$ measurement rate. The true image has relative error $0.09$ to a rank-15 image, and rank $r=15$ is used in the algorithm. Left: Original image. Center: Relative reconstruction error per iteration. Right: Reconstructed image after $200$ iterations.}\label{fig:truekopen}
\end{figure}

\begin{figure}[h!]
\includegraphics[width=2.1in]{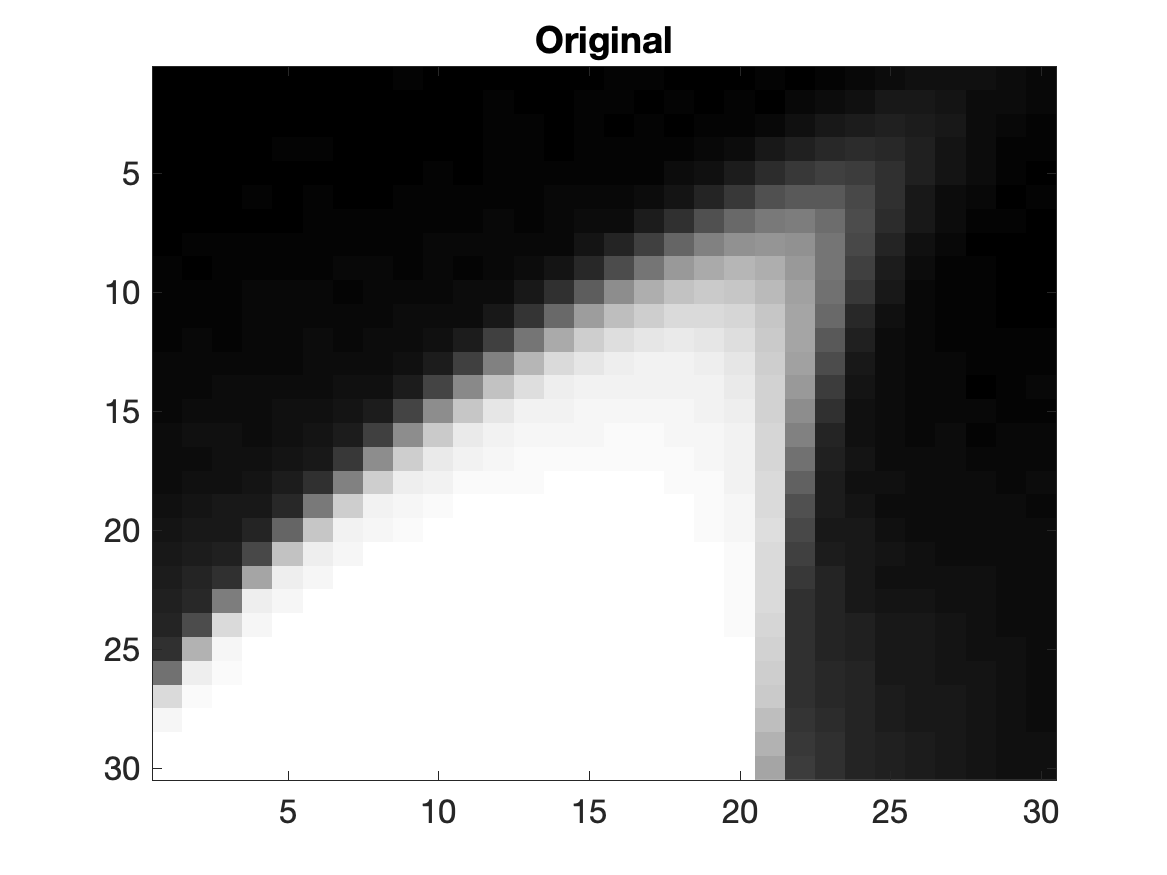}$\;$\includegraphics[width=2.1in]{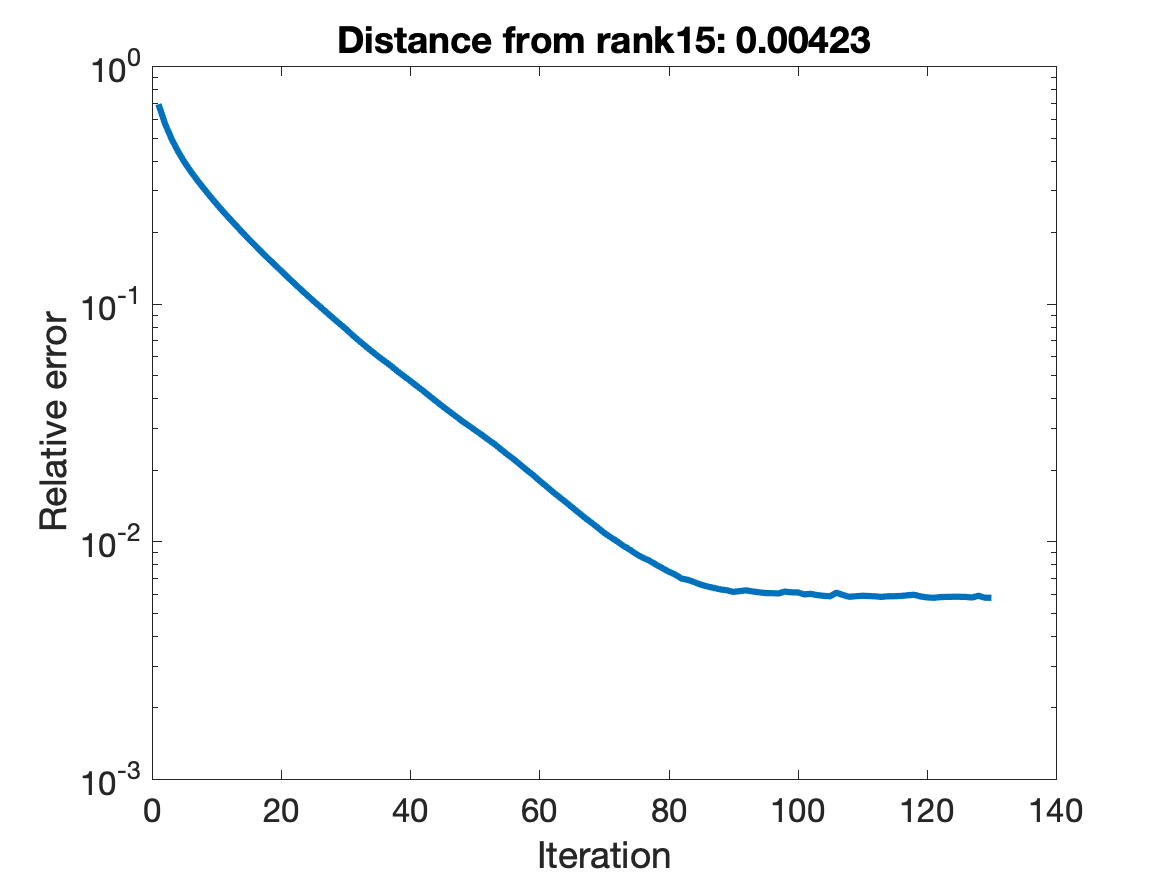}$\;$\includegraphics[width=2.1in]{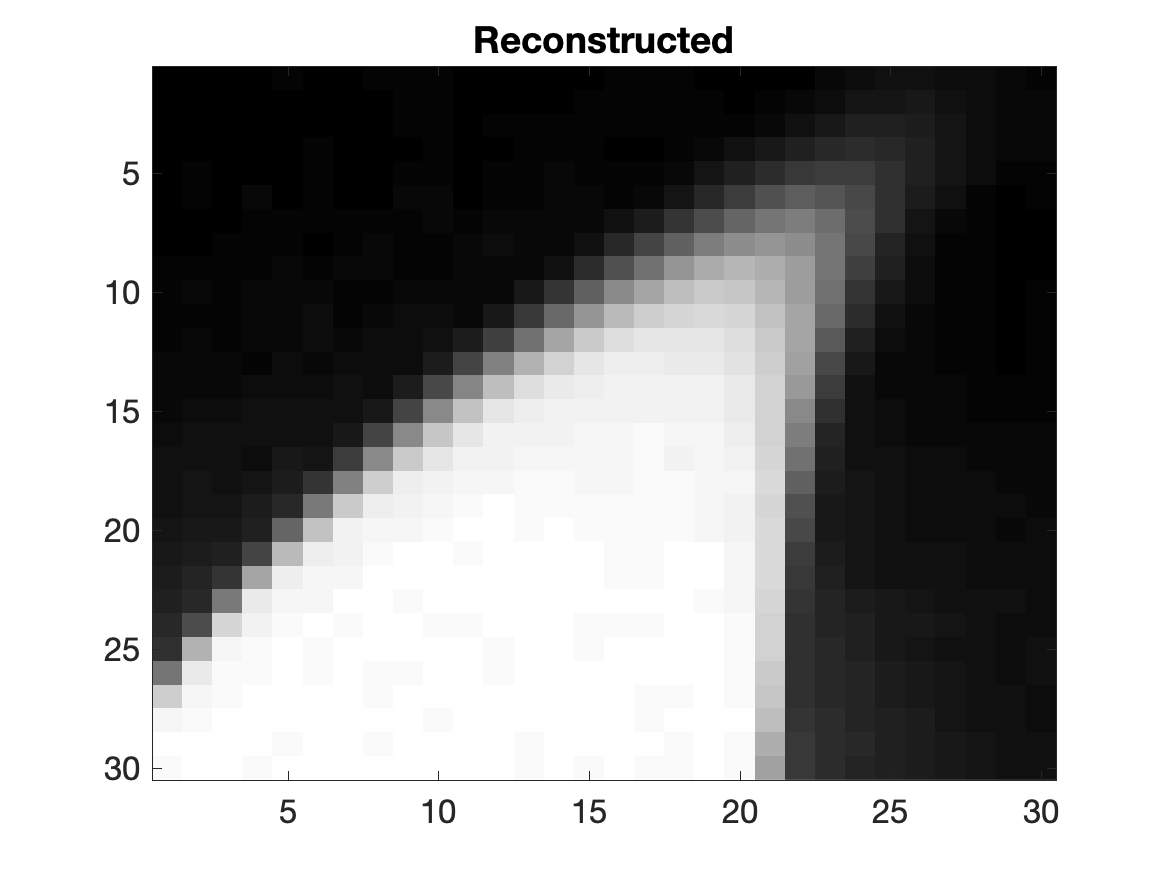}
\caption{Candle video of size $30\times 30\times 10$, $80\%$ measurement rate, rank $r=15$. Left: Original image corresponding to the first frame. Center: Relative reconstruction error per iteration. Right: Reconstructed first frame after $130$ iterations. }\label{fig:candle}
\end{figure}

\section{Conclusion}\label{sec:conclude}
This work presents an Iterative Hard Thresholding approach to low CP-rank tensor approximation from linear measurements. We show that the proposed algorithm converges to a low-rank tensor when a linear operator satisfies an RIP-type condition for low rank tensors (TRIP). In addition, we prove that Gaussian measurements satisfy the TRIP condition for low CP-rank signals. Our numerical experiments not only verify our theoretical findings but also highlight the potential of the proposed method. Future directions for this work include extensions to stochastic iterative hard thresholding~\cite{nguyen2017linear} and utilizing different types of tensor measurement maps~\cite{georgieva2019greedy}.

\section*{Acknowledgements}
This material is based upon work supported by the National Security Agency under Grant No. H98230-19-1-0119, The Lyda Hill Foundation, The McGovern Foundation, and Microsoft Research, while the authors were in residence at the Mathematical Sciences Research Institute in Berkeley, California, during the summer of 2019. In addition, Needell was funded by NSF CAREER DMS $\#1348721$ and NSF BIGDATA DMS $\#1740325$. Li was supported by the NSF grants CCF-1409258, CCF-1704204, and the DARPA Lagrange Program under ONR/SPAWAR contract N660011824020. Qin is supported by the NSF grant DMS-1941197.

\bibliographystyle{ieeetr}
\bibliography{overall}

\end{document}